\documentclass{conm-p-l}

\newtheorem{theorem}{Theorem}[section]
\newtheorem{proposition}[theorem]{Proposition}

\theoremstyle{definition}
\newtheorem{definition}[theorem]{Definition}

\theoremstyle{remark}
\newtheorem{remark}[theorem]{Remark}

\numberwithin{equation}{section}

\usepackage{amsfonts, amsmath, amssymb, amsgen, amsthm, amscd,latexsym}
\usepackage[all]{xy}

\def\Hom{\mathop{\rm Hom}\nolimits}
\def\Tot{\mathop{\rm Tot}\nolimits}

\def\Cb{{\mathbb C}}

\def\a{\alpha}

\def\D{\Delta}

\def\om{\omega}

\def\s{\sigma}

\def\tPsi{\tilde\Psi}
\def\ot{\otimes}
\def\ra{\rightarrow}

\def\al{>\hspace{-4pt}\vartriangleleft}

\def\hd{\overset{\ra}{\partial}}
\def \vd{\uparrow\hspace{-4pt}\partial}
\def\hs{\overset{\ra}{\sigma}}
\def \vs{\uparrow\hspace{-4pt}\sigma}
\def\hta{\overset{\ra}{\tau}}
\def \vta{\uparrow\hspace{-4pt}\tau}
\def\hb{\overset{\ra}{b}}
\def \vb{\uparrow\hspace{-4pt}b}

\def\hB{\overset{\ra}{B}}
\def \vB{\uparrow\hspace{-4pt}B}
\def\p{\partial}

\def\0D{\Delta^{(0)}}
\def\1D{\Delta^{(1)}}
\def\Db{\blacktriangledown}

\def\td{\tilde}

\def\ni{\noindent}

\def\build#1_#2^#3{\mathrel{
\mathop{\kern 0pt#1}\limits_{#2}^{#3}}}
\newcommand{\ps}[1]{~\hspace{-4pt}_{^{(#1)}}}
\newcommand{\ns}[1]{~\hspace{-4pt}_{_{{<#1>}}}}
\newcommand{\sns}[1]{~\hspace{-4pt}_{_{{<\overline{#1}>}}}}

\def\odots{\ot\cdots\ot}

\numberwithin{equation}{section}

\def\a{\alpha}

\def\om{\omega}
\def\s{\sigma}

\def\ve{\varepsilon}

\def\vp{\varphi}

\def\D{\Delta}

\def\ot{\otimes}
\def\part{\partial}

\def\ra{\rightarrow}

\def\text{\hbox}

\def\ot{\otimes}

\def\ra{\rightarrow}

\def\Hom{\mathop{\rm Hom}\nolimits}

\def\Id{\mathop{\rm Id}\nolimits}

\def\build#1_#2^#3{\mathrel{
\mathop{\kern 0pt#1}\limits_{#2}^{#3}}}

\numberwithin{equation}{section}
\date{ \today}

\begin{document}
\title{Cup products in Hopf cyclic cohomology with coefficients in  contramodules}
\author{Bahram Rangipour}

\address{ Department of Mathematics and
Statistics,University of New Brunswick, Fredericton, NB,  CANADA, E3B 5A3}

\email{bahram@unb.ca}


\dedicatory{Dedicated  to Henri Moscovici, with  highest admiration and appreciation .}
 \begin{abstract}
\ni We use stable anti Yetter-Drinfeld   contramodules to
 improve  the cup products  in Hopf cyclic cohomology.
The improvement  fixes the lack of functoriality  of the  cup products previously defined and show that the cup products are sensitive to the coefficients.
  \end{abstract}

\maketitle

\section{Introduction}
Hopf cyclic cohomology was invented by Alain Connes and Henri Moscovici as a computational tool for computing the index cocycle of the hypoelliptic operators on manifolds \cite{CM98}. One of the object of the theory was to study the cyclic cocycles generated by a  symmetric system, in the sense of  noncommutative geometry, which is usually given by an action or a coaction of a Hopf algebra on an algebra or a coalgebra.  The main tool for  transferring  such  cocycles  to the cyclic complex of algebras is  a  characteristic map  defined in \cite{CM98}. The characteristic map is based on an invariant trace on the algebra of functions on the manifold in question. However in many  situations  the invariant trace does not exist, for example see \cite{CM05}. For such cases the  invariant cyclic cocycles   play the role of  invariant trace and one defines a higher version of the characteristic map \cite{MC02,AG02}.  By the generalization of Hopf cyclic cohomology \cite{HKRS04-1,HKRS04-2} that allows one to take advantage of coefficients for Hopf cyclic cohomology,  the invariant cyclic cocycles are  understood as examples of Hopf cyclic  cocycles.  As a result,  one generalizes the characteristic map to a cup product \cite{KR05}. Similarly, the ordinary cup product in algebras  was also generalized to another type of cup product in Hopf cyclic cohomology by replacing cycles and their characters with twisted cycles and their twisted characters.
   In \cite{BR08,AK08}, by  a straight application of cyclic Eilenberg-Zilber theorem ( c.f.  \cite{KR04,GJ}),  the cup products was reconstructed and simplified. Finally, it is shown that all cup products defined in \cite{KR05,BR08,AK08,AG02} are the same in the level of cohomology \cite{AK10}.

   The suitable coefficients for Hopf cyclic cohomology  mentioned above is called stable anti Yetter-Drinfeld (SAYD) module \cite{HKRS04-2, JS}.
   It has both module and comodule structure, over the Hopf algebra in question, with two compatibilities made of composition of action and coaction.
    However it is proved that Hopf cyclic cohomology works with a generalization of  SAYD  modules called SAYD contramodules \cite{TB08}.
 Contramodules for a coalgebras was introduced in \cite{EM65}.  A right  contramodule of a coalgebra $C$ is a vector space $\mathcal{M}$ together with a $\Cb$-linear map $\alpha : \hom(C, \mathcal{M})\ra \mathcal{M}$ makes the   diagrams \eqref{diagram} commutative.

An SAYD contramodule  $\mathcal{M}$  is a module and contramodule together with  two compatibilities made of $\a$ and the action of $H$ on $\mathcal{M}$.  As an example if  $M$ is a SAYD module over $H$ then $\hom_k(M,\Cb)$ is an SAYD contramodule over $H$.

In this paper, building on the methods we developed in \cite{BR08}, we generalize the cup products defined in the same paper by using SAYD contramodules coefficients. By Theorem \ref{cup-general-1} and Theorem \ref{cup-theorem-2} we  show that the cup products is sensitive to  coefficients.  In Section 2 we recall Hopf cyclic cohomology with coefficients in SAYD modules and contramodules.    In Section 3 we  define the cup products for  compatible pair of SAYD modules and contramodules. Here a compatible pair reads a pair of SAYD module and contramodule  endowed with a pairing with values in the ground field and compatible  with respect to actions and coactions. Finally, in Section 4 we generalize the results of Section 3 for  arbitrary  coefficients without any compatibility between them. The range of new cup products are  ordinary cyclic cohomology of algebras with coefficients in  vector spaces.

In this note a Hopf algebra is denoted by a sextuple  $(H,\mu,\eta,\D,\ve, S)$, where $ \mu$, $\eta$, $\D,\ve$, and $S$ are  multiplication, unit, comultiplication, counit,  and antipode respectively. We use the Sweedler notation for comultiplications and coactions  i.e., for coalgebras  we use $\D(c)=c\ps{1}\ot c\ps{2}$, for comodules we use $\Db(a)=a\ns{0}\ot a\ns{1}$ and for coefficients we use  $\Db(m)=m\sns{-1}\ot m\sns{0}$. All algebras, coalgebras and Hopf algebras are over the field of complex numbers  $\Cb$. The unadorn tensor product $\ot$ reads  $\ot_\Cb$.

\medskip

We would like to thank  Tomasz Brzezi\'{n}ski for  Remark \ref{remark}. We are also grateful of the referee for his carefully reading the manuscript and his   valuable comments.

 \section{Hopf cyclic cohomology with coefficients}
\subsection{ Stable anti Yetter-Drinfeld-module}

For the  reader's convenience,  we  briefly   recall the definition of Hopf cyclic cohomology of coalgebras and algebras under the symmetry of
Hopf algebras  with coefficients in SAYD  modules \cite{HKRS04-1}, and with coefficients in SAYD contramodules \cite{TB08}.

Let us recall the definition of SAYD modules over a Hopf algebra from \cite{HKRS04-2}. Given a Hopf algebra  $H$, we say  that
 $M$ is a right-left SAYD module over $H$ if
$M$ is a right module and left module over $H$ with the following compatibilities.
\begin{align}\label{SAYD}
&\Db_M(m\cdot h)= S(h\ps{3})m\sns{-1}h\ps{1}\ot m\sns{0}\cdot h\ps{2}\\
&m\sns{0}\cdot m\sns{-1}=m.
\end{align}
The other three  flavors i.e,  left-left, left-right, and right-right are defined similarly \cite{HKRS04-2}.

Let $C$ be a $H$-module coalgebra, that is a coalgebra endowed with an action, say from left,
 of $H$ such that its comultiplication and counit are $H$-linear, i.e,
\begin{align}
\D(h\cdot c)=h\ps{1}\cdot c\ps{1}\ot h\ps{2}\cdot c\ps{2}, \quad \ve(h\cdot c)=\ve(h)\ve(c).
\end{align}

 Having the  datum $(H, C, M)$, where $C$ is an $H$-module coalgebra and $M$ an right-left SAYD over $H$,
 one defines in \cite{HKRS04-1} a  cocyclic module $\{C^n_H(C,M), \p_i,\s_j,\tau \}_{n\ge 0}$   as follows.

 \begin{equation}\label{cyclic-coalgebra-0}
  C^n_H(C,M):=M\ot_H C^{\ot n+1}, \quad n\ge 0,
 \end{equation}
 with the following cocyclic structure,
\begin{align}\label{cyclic-coalgebra-1}
&\p_i:C^n_H(C,M)\ra C^{n+1}_H(C,M), & 0\le i\le n+1\\\label{cyclic-coalgebra-2}
& \s_j: C^n_H(C,M)\ra C^{n-1}_H(C,M), & 0\le j\le n-1,\\\label{cyclic-coalgebra-3}
& \tau:C^n_H(C,M)\ra C^n_H(C,M),
\end{align}

defined explicitly as follows, where we abbreviate $\td c:=c^0\odots c^n$,
 \begin{align}\label{cyclic-coalgebra-4}
& \p_i(m\ot_H \td c )=m\ot_H c^0\odots \Delta(c^i)\odots c^n,
\\\label{cyclic-coalgebra-5}
 &\p_{n+1}(m\ot_H \td c)=m\sns{0}\ot_H c^0\ps{2}\ot c^1\odots c^n\ot
  m\sns{-1}\cdot c^0\ps{1},\\\label{cyclic-coalgebra-6}
 &\sigma_i(m\ot_H \td c)=m\ot_H c^0\odots \epsilon(c^{i+1})\odots c^n,
 \\\label{cyclic-coalgebra-7}
 &\tau(m\ot_H \td c)=m\sns{0}\ot_H c^1\odots c^n\ot m\sns{-1}\cdot c^0.
 \end{align}
It is checked in  \cite{HKRS04-1} that  the above graded module  defines a cocyclic module.

Similarly  an algebra  which is a $H$-module and its algebra structure is $H$-linear is called $H$-module algebra.  In other words,  for any $a,b\in A$ and any $h\in H$ we have
\begin{equation}
h\cdot(ab)=(h\ps{1}\cdot a)( h\ps{2}\cdot b), \quad h\cdot 1_A=\ve(h)1_A.
\end{equation}

Let $A$ be a  $H$-module algebra. One endows $M\ot A^{\ot n+1}$ with the diagonal action of $H$ and forms $C^n_H(A,M):=\Hom_H(M\ot A^{\ot n+1}, \Cb)$ as the space of
$H$-linear maps. It is checked in \cite{HKRS04-1} that the following defines a cocyclic module structure on $C^n(A,M)$.
\begin{align*}
&(\p_i \vp)(m\ot \td{a}) = \vp(m\ot a^0\ot \cdots
\ot a^i a^{i+1}\ot\cdots \ot a^{n+1}),\\
&  (\p_{n+1}\vp)(m\ot \td{a}) =\vp(m\sns{0}\ot
(S^{-1}(m\sns{-1})\cdot a^{n+1})a^0\ot a^1\ot\cdots \ot a^{n}),\\
&(\s_i\vp)(m\ot \td a) = \vp(m\ot a^0 \ot \cdots
\ot a^i\ot 1 \ot \cdots \ot a^{n-1}), \\
&(\tau\vp)(m\ot \td a) = \vp(m\sns{0}\ot S^{-1}(m\sns{-1})\cdot a^n\ot a^0\ot \cdots \ot a^{n-1}).
\end{align*}

The cyclic cohomology of this cocyclic module is denoted by $HC^\ast_H(A,M)$.

\medskip

An algebra is called a  $H$-comodule algebra if  it is a $H$ comodule and its algebra structure are $H$ colinear, which means that
\begin{equation}
(h\cdot a)\ns{0}\ot (h\cdot a)\ns{1}= h\ps{1}\cdot a\ns{0}\ot h\ps{2}\cdot a\ns{1}.
\end{equation}

Similar to the other case, one
defines  $^H C^n(A,M)$ to be the space of all colinear maps from $A^{\ot n+1}$ to $M$. One checks that the following defines a cocyclic module
structure on $^H C^n(A,M)$.

\begin{align}\label{cyclic-algebra-comodule-1}
&(\p_i\vp)(\td a)=\vp( a^0\ot \cdots
 \ot a^i a^{i+1}\ot\cdots \ot a^{n+1}),\\\label{cyclic-algebra-comodule-2}
 &(\p_{n+1}\vp)(\td a)= \vp(a^{n+1}\ns{0}a^0\ot a^1\cdots \ot
a^{n-1}\ot a^{n})\cdot a^{n+1}\ns{-1},\\\label{cyclic-algebra-comodule-3}
& (\s_i \vp)( \td a)= \vp(a^0 \ot \cdots \ot a^i\ot
 1 \ot \cdots \ot a^{n-1}),\\\label{cyclic-algebra-comodule-4}
&(\tau\vp )(a^0\ot\cdots\ot a^n)=\vp( a^n\ns{0}\ot
  a^0\ot\cdots \ot a^{n-1}\ot a^{n-1})\cdot a^{n}\ns{-1}.
\end{align}
The cyclic cohomology of this cocyclic module is denoted  by $^H HC^\ast(A,M)$.

\subsection{SAYD contramodule}
Let  us recall SAYD contramodules from \cite{TB08}.  A  right contramodule of a coalgebra   $H$ is a vector space
 $\mathcal{M}$ together with a $\Cb$-linear map $\alpha : \Hom(H, \mathcal{M})\to \mathcal{M} $ making the following diagrams commutative
\begin{equation}\label{diagram}
\xymatrix{
\Hom (H,  \Hom(  H, \mathcal{M}) \ar[rrrr]^-{\Hom(H,  {\a})}\ar[d]_\Theta &&&& \Hom (H,  \mathcal{M}) \ar[d]^{\a} \\
\Hom (H\ot H, \mathcal{M}) \ar[rr]^-{\Hom(\D,  \mathcal{M})} && \Hom(H,  \mathcal{M}) \ar[rr]^-{\a} && \mathcal{M} ,}
$$
$$
\xymatrix{\Hom(  \Cb,  \mathcal{M}) \ar[rr]^-{\Hom(  \ve,  \mathcal{M})} \ar[rd]_\simeq && \Hom( H,  \mathcal{M})\ar[dl]^{\a} \\
& \mathcal{M} , & }
\end{equation}

where $\Theta$ is the standard isomorphism given by $\Theta(f)(h\ot h') = f(h)(h')$.
\begin{definition}[\cite{TB08}]
 A  left-right anti-Yetter-Drinfeld (AYD) contramodule $\mathcal{M}$ is a left $H$-module (with the action denoted by a dot) and a right
  $H$-contramodule with the
structure map $\alpha$, such that, for all $h\in H$ and $f\in \Hom(H, \mathcal{M})$,
$$
h\!\cdot\! \alpha(f) = \alpha\left( h\ps{2}\!\cdot\! f\left( S(h\ps{ 3})(-) h\ps{1}\right)\right).
$$
$M$ is said to be {\em stable}, provided that, for all $m\in \mathcal{M}$, $\alpha(r_m) = m$,
 where $r_m: H\to \mathcal{M}$, $h\mapsto h\!\cdot\! m$.

\end{definition}

We refer the reader to \cite{TB08} for more details on SAYD contramodules.
If $M$ is an AYD module, then its dual $\mathcal{M}=M^*$ is an
 AYD contramodule (with the sides interchanged) and  SAYD
 modules correspond to SAYD contramodules. For example, let $M$ be  a right-left AYD
 module  \eqref{SAYD}, the  dual vector space $\mathcal{M}= M^*$ is a right  $H$-module by $ m \ot h \mapsto  m\cdot h$,
$$
(h\!\cdot\! f) (m) =f(m\!\cdot\! h),
$$
for all $h\in H$, $f\in \mathcal{M} =\Hom(M, \Cb)$ and $m\in M$, and a right
 $H$-contramodule with the structure map $\alpha(f)(m) = f(m\ns{-1})(m\ns{0})$, $f\in
\Hom(H, \mathcal{M})$, and $m\in M$ \cite{TB08}.

Let $A$ be a left $H$-module algebra and $\mathcal{M}$ be a left-right  SAYD contramodule over $H$.  We let $C^n_H(A,\mathcal{M})$ to be the space of left
$H$-linear maps

\begin{equation} \label{cycli-contra-0}
\Hom_H( {A^{\otimes n+1}},  \mathcal{M}),
\end{equation}
 and, for all $0\leq i,j\leq n$, define $ \p_i: C^{n-1}_H(A,M)\to C^{n}_H(A,\mathcal{M})$,
$\sigma_j: C^{n+1}_H(A,\mathcal{M})\to C^{n}_H(A,\mathcal{M})$, $\tau: C^n_H(A,\mathcal{M})\to C^n_H(A,\mathcal{M})$, by
\begin{align}\label{cycli-contra-1}
&\p_i(\vp)(a^0\odots a^{n}) = \vp(a^0 \odots  a^ia^{i+1}\odots  a^{n}), \qquad 0\leq i < n, \\ \label{cycli-contra-2}
&\p_{n}(\vp)(a^0\odots a^{n}) = \alpha\left(\vp\left(\left(S^{-1}(-)\!\cdot\! a^n\right) a^0\ot a^1 \odots  a^{n-1}\right)\right),\\\label{cycli-contra-3}
&\sigma_j (\vp) (a^0\odots a^{n}) = \vp(a^0\odots a^j\ot 1_A\ot a^{j+1}\odots a^{n}),\\\label{cycli-contra-4}
&\tau(\vp)(a^0\odots a^{n}) = \alpha\left(\vp\left(S^{-1}(-)\!\cdot\! a^n\ot a^0\odots  a^{n-1}\right)\right).
\end{align}
 It is shown in \cite{TB08} that the above operators define a  cocyclic module on $C^\ast_H( A,\mathcal{M})$.  We denote the cyclic cohomology of
 $C^\ast_H(A, \mathcal{M})$ by $HC^\ast_H(A,\mathcal{M})$. For  $\mathcal{M}=M^\ast$, where $M$ is a SAYD  module over $H$, it is easy to see that  $C^\ast_H(A,\mathcal{M})$$\simeq
 C^\ast_H(A,M)$.

 Indeed let $M$ be a right-left SAYD module and $\mathcal{M}:=\Hom(M, \Cb)$ be the corresponding right-left SAYD contramodule. We define  the following maps
 \begin{align*}
& \mathcal{I}: C^n_H(A, M)\ra C^n_H(A,\mathcal{M}), \quad \mathcal{J}: C^n_H(A, \mathcal{M})\ra C^n_H(A,M), \\
 &\mathcal{I}(\phi)(a^0\odots a^n)(m)=\phi(m\ot a^0\odots a^n), \\
 &\mathcal{J}(\phi)(m\ot a^0\odots a^n)=\phi(a^0\odots a^n)(m).
 \end{align*}
 \begin{proposition}
 The above map $\mathcal{I}$ is an isomorphism of cocyclic modules.
 \end{proposition}
 \begin{proof} It is obvious that $\mathcal{I}$ and $\mathcal{J}$ are inverse to one another. We shall check that $\mathcal{I}$  commutes with cyclic structures. It is easy to see that faces, except possibly the very last one,  and degeneracies commutes with $\mathcal{I}$. So it is suffices to check that  $\mathcal{I}$ commutes with the cyclic operators. Indeed,
 \begin{align*}
 &\mathcal{I}\circ\tau(\phi)(a^0\odots a^n)(m)=\tau(\phi)(m\ot a^0\odots a^n)\\
 &=\phi(m\sns{0}\ot S^{-1}(m\sns{-1})\cdot a^n\ot a^0\odots a^{n-1})\\
 &=\mathcal{I}(\phi)(S^{-1}(m\sns{-1})\cdot a^n\ot a^0\odots a^{n-1})(m\sns{0})\\
 &=\tau\circ\mathcal{I}(\phi)(a^0\odots a^n)(m).
 \end{align*}
 \end{proof}
\section{Cup products in Hopf cyclic cohomology}
In this section we use the same strategy as in \cite{BR08,AK08}
to generalize the cup products constructed in the same references.
Via these new cup products one has the luxury to construct cyclic cocycles by using a
compatible pair of  SAYD modules and contramodules rather than only a SAYD module.
\subsection{Module algebras paired with   module coalgebras}
Let $A$ be an $H$ module algebra and  $C$ be a $H$ module coalgebra acting on
$A$ in the sense that there is a map
\begin{equation}\label{0}
C\ot A\ra A,
\end{equation}
such that for any $h\in H$, any $c\in C$ and any $a,b\in A$ one has
\begin{align}\label{1}
 &(h\cdot c)\cdot a= h\cdot (c\cdot a)\\\label{2}
 &c\cdot (ab)=(c\ps{1}\cdot a)(c\ps{2}\cdot b)\\\label{3}
&c(1)=\epsilon(c)1
\end{align}
 One constructs a convolution algebra $B=\Hom_H(C,A)$,
which is the algebra of all $H$-linear  maps from $A$ to $C$.
 The unit of this algebra is given by $\eta\circ\epsilon$,
  where $\eta:\Cb\ra A$ is the unit of $A$.
  The multiplication of $f,g\in B$ is given by
\begin{equation}\label{4}
(f\ast g)(c)=f(c\ps{1})g(c\ps{2})
\end{equation}
\begin{definition}
Let $(\mathcal{M}, \a)$ be a left-right SAYD contramodule and $N$ be a
right-left SAYD module over $H$. We call $(N, \mathcal{M})$ compatible if there is a pairing
between $\mathcal{M}$ and $N$ such that
\begin{align}\label{compatibility1}
&  < n\cdot h\;\mid\; m>= < n\;\mid\; h\cdot m >,\\\label{compatibility2}
&  <  n \;\mid\;\a(f)>=<n\sns{0} \;\mid\; f(n\sns{-1})>,
 \end{align}
 for all $m\in \mathcal{M}$, $n\in N$, $f\in \Hom(H,\mathcal{M})$, and $h\in H$.
\end{definition}

 Let $(N,\mathcal{M})$ be compatible as above. We have the following cocyclic modules defined in
\eqref{cycli-contra-0} \dots \eqref{cycli-contra-4}, and \eqref{cyclic-coalgebra-0}
 \dots \eqref{cyclic-coalgebra-4} respectively.

\begin{equation}(C^\ast_H(A,\mathcal{M}), \p_i,\s_j,\tau), \quad \text{and}\quad (C^\ast_H(C,N),
\p_i,\s_j,\tau).\end{equation}
We define a new bicocyclic module by tensoring these cocycle module over $\Cb$.
The new bigraded  module has in its bidegree  $(p,q)$
\begin{equation}\label{c**}C_{\rm a-c}^{p,q}:=\Hom_H(A^{\ot p+1},\mathcal{M})\ot
(N\ot_H C^{\ot q+1}),\end{equation} with
horizontal structure $\hd_i=\Id\ot \p_i$, $\hs_j=\Id\ot \s_j$,
 and $\hta=\Id\ot \tau$ and vertical structure $\vd_i=\p_i\ot\Id$,
  $\vs_j=\s_j\ot\Id$, and
$\vta=\tau\ot \Id$. Obviously $(C^{p,q}_{\rm a-c},\hd,\hs,\hta,\vd,\vs,\vta)$
  defines a bicocyclic module.

Now let us define  the  map
\begin{align}\label{acpsi}
&\Psi:D^q(C_{\rm a-c}^{\ast,\ast})\ra \Hom(B^{\ot q+1},\Cb),\\\notag &
\Psi(\phi\ot (n\ot c^0\odots c^q))(f^0\odots f^q)=
< n \;\mid\;  \phi(f^0(c^0)\odots f^q(c^q)) >.
\end{align}
Here $D(C^{\ast,\ast}_{\rm a-c})$ denotes the diagonal of the bicocyclic module  $C^{\ast,\ast}_{\rm a-c}$. It is a cocyclic module whose $q$th component is $C^{q,q}$ and its cocyclic structure morphisms are  $\p_i:=\hd_i\circ\vd_i$, $\s_j:=\hs_j\circ\vs_j$, and $\tau:=\hta\circ\vta$.
\begin{proposition} \label{cyclic-map}
The map $\Psi$ is  a well-defined map of cyclic modules.
\end{proposition}
\begin{proof}
First let us show that $\Psi$ is well-defined. Indeed,
by using the facts that $\mathcal{M}$ and $ N$ are compatible,
 $f^i$ are  $H$-linear, $\phi$ is  equivariant  and
\eqref{1} holds, we see that,
\begin{align*}
&\Psi(\phi\ot (n\ot h\ps{1}c^0\odots h\ps{n+1}c^n))(f^0\odots f^n)\\
&=< n \;\mid\;  \phi(f^0(h\ps{1}\cdot c^0)\odots f^n(h\ps{n+1}
\cdot c^n))\\
&=<n \;\mid\; h\cdot\phi(f^0(c^0)\odots f^n(c^n))>\\
&=<nh \;\mid\; \phi(f^0(c^0)\odots f^n(c^n))>\\
&=\Psi(\phi\ot (n\cdot h\ot c^0\odots c^n))(f^0\odots f^n).
\end{align*}
 Next, we show that $\Psi$  commutes with cocyclic structure morphisms.
  To this end,  we need only to show the commutativity of $\Psi$ with zeroth cofaces, the last
  codegeneracies and the cyclic operators because these  operators generate all cocyclic structure morphisms.
We check it only  for the cyclic operators and leave the rest to the reader.
 Let $\tau_B$ denote the cyclic operator of the ordinary cocyclic module of the algebra $B$.
\begin{align*}
&\Psi(\tau(\vp)\ot \tau(n\ot c^0\odots c^q))(f^0\odots f^q)\\
&=\Psi(\tau\vp\ot(n\sns{0}\ot c^1\odots c^q\ot n\sns{-1}c^0))(f^0\odots f^q)\\
&=< n\sns{0}  \;\mid\; \tau(\vp)(f^0(c^1)\odots f^{q-1}(c^n)\ot f^q(n\sns{-1}c^0))>\\
&=< n\sns{0}  \;\mid\; \alpha\left(\vp\left(S^{-1}(-)\!\cdot\! f^q(n\sns{-1} c^0)\ot f^0(c^1)
\odots f^{q-1}(c^n) \right)\right)> \\
&=<  n\sns{0} \;\mid\;\vp\left(S^{-1}(n\sns{-1})\!\cdot\! f^q(n\sns{-2}c^0)\ot
f^0(c^1)\odots f^{q-1}(c^n) \right)>\\
&=< n \;\mid\; \vp\left( f^q(c^0)\ot f^0(c^1)\odots f^{q-1}(c^q) \right)>\\
&=\tau_B\Psi(\vp\ot (n\ot c^0\odots c^q))(f^0\odots f^q).
\end{align*}
Here in the passage from fourth line to the fifth one  we use \eqref{compatibility2}.
\end{proof}

Let  $C:=\bigoplus_{p,q\ge 0} C^{p,q}$  be a bicocyclic module.  With $Tot(C)$ designating the total mixed complex
$\, Tot(C)^n=\bigoplus_{p+q=n} C^{p,q}$, we
denote by $\Tot(C)$ the associated \textit{normalized} subcomplex,
obtained by retaining only the elements annihilated by all degeneracy
operators. Its total boundary is $b_T+B_T$, with $b_T$ and $B_T$
defined as follows:
\begin{align} \notag
&\hb_p=\sum_{i=0}^{p+1} (-1)^{i}\hd_i,
&&\vb_q=\sum_{i=0}^{q+1} (-1)^{i}\vd_i, \\ \label{bT}
&b_T=\sum_{p+q=n}\hb_p+\vb_q,\\  \notag
&\hB_p=(\sum_{i=0}^{p-1}(-1)^{(p-1)i}\hta^i)\hs_{p-1}\hta,  &&\vB_q=
(\sum_{i=0}^{q-1}(-1)^{(q-1)i}\vta^i)\hs_{q-1}\vta,\\ \label{BT}
&B_T=\sum_{p+q=n}\hB_p+\vB_q.\
\end{align}

The total complex of a bicocyclic module $C$ is a mixed complex, i.e, $b_T^2=B_T^2=b_TB_T+B_Tb_T=0$. As a result
 its cyclic cohomology is well-defined. By means of the analogue of
 the Eilenberg-Zilber theorem for bi-paracyclic
modules~\cite{GJ,KR04}, the diagonal mixed complex
$(D(C), b_D, B_D)$ and the total mixed complex $(\Tot C,b_T, B_T)$ can be seen to
be quasi-isomorphic in both Hochschild and cyclic cohomology. Here
 $D(C):=\bigoplus_{q\ge 0} C^{q,q}$ is a cocyclic module and therefore
 a mixed complex with (co)boundaries,
\begin{align}
\begin{split}
 &b_D:= \sum_{i=0}^{q+1}(-1)^q\vd_i\circ\hd_i,\\
 &B_D:= \left(\sum_{i=0}^{q-1}(-1)^{(q-1)i}\hta^i\vta^i\right)\hs_{q-1}\vs_{q-1}\hta\vta.
 \end{split}
 \end{align}
At the level of Hochschild cohomology the quasi-isomorphism is
implemented by the Alexander-Whitney map
$\, AW:= \bigoplus_{p+q=n} AW_{p,q}: \Tot(C)^n\ra D(C)^n $,
\begin{align}\label{AW}
\begin{split}
&AW_{p,q}: C^{p,q}\longrightarrow C^{p+q,p+q} \\
&AW_{p,q}=(-1)^{p+q}\underset{q\,\text{times}}{\underbrace{\vd_0\vd_0\dots
\vd_0}}\hd_n\hd_{n-1}\dots \hd_{q+1} \, .
\end{split}
\end{align}
Using a standard homotopy operator $H$, this can be supplemented by
a cyclic Alexander-Whitney map
$AW':= AW\circ B\circ H:  D^n\ra \Tot(C)^{n+2}$, and thus
upgraded to an $S$-map $\overline{AW} = (AW, AW')$,
 of mixed complexes. The inverse quasi-isomorphims are
provided by the shuffle maps $\, Sh:= D(C)^n\ra \Tot(C)^n$, resp.
$\overline{Sh} = (Sh, Sh')$, which are discussed in detail in~\cite{GJ,KR04}.

 Let  $c$ be $(b,B)$ cocycle  in $\Tot(C)^n$, $n=p+q$.
 Hence the class of  $\overline{AW}(c)$ in $HC^n(D(C))$ is well defined.

Now we consider the inclusion  $\iota:A\ra B=\Hom_H(C,A)$,
 defined by $\iota(a)(c)=c\cdot a$.
We see that
 $\iota(ab)(c)= c\cdot(ab)=(c\ps{1}\cdot a)(c\ps{2}\cdot b)=(\iota(a)\ast\iota(b))(c),$
 and  $\iota(1_A)(c)=c\cdot 1_A =\ve(c)1_A=1_B(c).$
 Hence $\iota$ is an algebra map and in turn
 induces a map in the level
 of cyclic cohomology groups: $$\iota: HC^{\ast}(B)\ra HC^{\ast}(A). $$

\begin{theorem} \label{cup-theorem-1}Let $H$ be a Hopf algebra,  $A$ be an $H$-module algebra,  $C$ be an
 $H$-module coalgebra acting on $A$,  and   $(N, \mathcal{M})$ be a compatible pair of SAYD
  module and contramodule over $H$. Then $\hat\Psi:= \iota\circ\Psi \circ\overline{AW}$
 defines a cup product in the level of cyclic cohomology groups:
\begin{equation}\label{cup-cont-coalgebra}
\hat\Psi:= \iota\circ\Psi \circ\overline{AW}: HC^p_H(A,\mathcal{M})\ot HC^q_H(C,N)\ra HC^{p+q}(A).
\end{equation}
\end{theorem}
\begin{proof}
Let $[\phi]\in HC^p_H(A,\mathcal{M})$ and $[\om]\in HC^q_H(C,N)$. Without loss of
 generality one assumes that $\phi$ and $\om$ are both cyclic cocycles, i.e,
  $$\hb(\phi)=\vb(\om)=0, \quad  \hta(\phi)=(-1)^p\phi, \quad   \vta(\om)=(-1)^q\om.$$  This implies
 that $\phi\ot \om$ is a $(b,B)$ cocycle in $\Tot(C_{\rm a-c}^{\ast,\ast})^{p+q}$.
 Hence $\overline{AW}(\phi\ot \om)$ defines a class in
 $HC^{p+q}(D(C_{\rm a-c}^{\ast,\ast}))$.
 Finally,  since $\iota$ and $\Psi$ both are cyclic map, the transferred
  cochain   $\iota\circ\Psi(\overline{AW}(\phi\ot \om))$ defines a class in $HC^{p+q}(A)$.
\end{proof}

\subsection{ Module algebras paired with comodule algebras}
Let $H $ be a Hopf algebra, $A$ a   left $H $-module algebra,
 $B$ a left $H $-comodule algebra,  and $(N,\mathcal{M})$ be a compatible
  pair  of SAYD module and contramodule over $H$.
One  constructs a crossed product algebra whose underlying vector
 space is $A\ot B$ with the $1\al 1$ as its unit and the following
 multiplication:
\begin{equation}\label{crossed-algebra}
(a\al b)(a'\al b')= a\,(b\ns{-1}\cdot a')\al b\ns{0}b'
\end{equation}
Now consider  the two cocyclic modules
\begin{equation}(C^\ast_H (A,\mathcal{M}), \p_i,\s_j,\tau), \quad \text{and}\quad
 (^H  C^\ast(B,N), \p_i,\s_j,\tau)
 \end{equation}
 introduced in \cite{TB08} and \cite{HKRS04-1} respectively and are recalled
  in \eqref{cyclic-algebra-comodule-1}  \dots \eqref{cyclic-algebra-comodule-4}
   and \eqref{cycli-contra-0}\dots
\eqref{cycli-contra-4}. We  define a bicocyclic module
by tensoring these cocyclic modules over $\Cb$.
The  $(p, q)$-bidegree component  $C^{p,q}_{\rm a-a}$
of this new bicocyclic module is given by
 \begin{equation}\label{bicyclic-algebra}
  ^H  \Hom(B^{\ot q+1}, N) \ot \Hom_H ( A^{\ot p+1},\mathcal{M}),
 \end{equation}

  with horizontal structure morphisms $\hd_i=\Id\ot \p_i$,
$\hs_j=\Id\ot \s_j$, and $\hta=\Id\ot \tau$ and vertical structure morphisms
$\vd_i=\p_i\ot \Id$, $\vs_j=\s_j\ot \Id$, and $\vta=\tau\ot \Id$.
 Now we define  a new morphism
 \begin{equation}\Phi:D(C^{\ast,\ast}_{\rm a-a})^n\ra C^n( A\al B),\end{equation}
  define by
\begin{align}\label{abpsi}
& \Phi(\psi\ot\phi)(a^0\al b^0\odots a^n\al b^n)=\\\notag
&=<\psi(b^0\ns{0}\odots b^n\ns{0})\;\mid\; \phi(  S^{-1}(b^0\ns{-1}\cdots
 b^n\ns{-1})\cdot a^0\ot\cdots\\&~~~~~~~~~~~~~~~~~~~~~~~~~~~~~~~~~~~
 ~~~~~~~~~~~\dots\ot S^{-1}(b^n\ns{-n-1})\cdot a^n)>.
\end{align}
\begin{proposition}\label{cyclic}
The map $\Phi$ defines a  cyclic map between  the diagonal of
$C_{\rm a-a}^{\ast,\ast}$ and the cocyclic module $C^\ast(A\al B)$.
\end{proposition}
\begin{proof}
We  show that $\Phi$  commutes with the cyclic structure morphisms.
 We shall  check it for the first  face operator and the   cyclic
operator and leave the rest  to the reader. Let us denote the cyclic structure
 morphisms of the algebra $A\al B$ by
$\p_i^{A\al B}, \s_j^{A\al B}$ and $\tau^{A\al B}$.
First we show that $\Phi$ commutes with the zeroth cofaces.
\begin{align*}
&\Phi(\hd_0\vd_0 (\psi\ot\phi))(a^0\al b^0\odots a^{n+1}\al b^{n+1})\\
&=\Phi (\p_0\phi\ot\p_0\psi))(a^0\al b^0\odots a^{n+1}\al b^{n+1})\\
& =<\p_0\psi(b^0\ns{0}\odots b^{n+1}\ns{0})\;\mid\;   \p_0\phi(  S^{-1}
(b^0\ns{-1}\cdots b^{n+1}\ns{-1})\cdot a^0\ot\cdots\\
&\hspace{7.3cm}\cdots\ot  S^{-1}(b^{n+1}\ns{-n-2})\cdot a^{n+1}  )>\\
&=<\psi(b^0\ns{0}b^1\ns{0}\odots b^{n+1}\ns{0})     \\
&\hspace{2cm}\;\mid\; \phi( S^{-1}(b^0\ns{-1}\cdots b^{n+1}\ns{-1})\cdot a^0S^{-1}
(b^1\ns{-2}\cdots b^{n+1}\ns{-2})\cdot a^1\ot\cdots\\
&\hspace{7.3cm}\cdots\ot  S^{-1}(b^{n+1}\ns{-n-2})\cdot a^{n+1}  )>\\
&=<\psi(b^0\ns{0}b^1\ns{0}\odots b^{n+1}\ns{0}) \\
&  \hspace{2cm} \;\mid\; \phi( S^{-1}(b^0\ns{-1}b^1\ns{-1}\cdots b^{n+1}\ns{-1})
\cdot ( a^0 b^0\ns{-2}\cdot a^1)\ot\cdots\\
&\hspace{7.3cm}\cdots\ot  S^{-1}(b^{n+1}\ns{-n-1})\cdot a^{n+1}  )>\\
&=\Phi(\psi\ot\phi)(a^0 b^0\ns{-1}a^1\al b^0\ns{0}b^1\ot  a^2\al
 b^2\odots a^{n+1}\al b^{n+1})\\
&=\p_0^{A\al B}\Phi(\psi\ot \phi)(a^0\al b^0\odots a^{n+1}\al b^{n+1}).
\end{align*}
Now we show that $\Phi$ commutes with cyclic operators.
\begin{align}
\begin{split}\label{5}
&\Phi(\hta\vta(\psi\ot\phi))(a^0\al b^0\odots a^{n}\al b^{n}) \\
&=\Phi(\tau\psi\ot\tau\phi)(a^0\al b^0\odots a^{n}\al b^{n})\\
&=<\tau\psi(b^0\ns{0}\odots b^n\ns{0})\;\mid\; \tau\phi( S^{-1}(b^0\ns{-1}\cdots
 b^n\ns{-1})a^0\ot\cdots\\
 &~~~~~~~~~~~~~~~~~~~~~~~~~~~~~~~~~~~~~~~~~~~~~~~~~\cdots\ot S^{-1}(b^n\ns{-n-1})a^n)>\\
&=<\psi(b^n\ns{0}\ot b^0\ns{0}\ot\cdots\\
&\cdots\ot b^{n-1}\ns{0})\cdot b^n\ns{-1}\;\mid \;
\a(\phi(S^{-1}(-) S^{-1}(b^n\ns{-n-2})\cdot a^n\ot \\
&S^{-1}(b^0\ns{-1}\cdots b^{n-1}\ns{-1}b^n\ns{-2})\cdot a^0
\odots S^{-1}(b^{n-1}\ns{-n}b^n\ns{-n-1})\cdot a^{n-1}))>.
\end{split}
 \end{align}
 Using \eqref{compatibility2}, and the fact that $\psi$ is $H$-colinear,  one has:
 \begin{align}\label{6}
 \begin{split}
&\eqref{5}=<[\psi(b^n\ns{0}\ot b^0\ns{0}\odots b^{n-1}\ns{0}) \cdot b^n\ns{-1}]
\ns{0} \;|\;\\ &\phi(
 S^{-1}([\psi(b^n\ns{0}\ot b^0\ns{0}\odots b^{n-1}\ns{0}) \cdot b^n\ns{-1}]
\ns{-1})(S^{-1}(b^n\ns{-n-2})\cdot a^n)\ot\\
&\ot S^{-1}(b^0\ns{1}\cdots b^{n-1}\ns{-1}b^n\ns{-2})\cdot a^0\odots S^{-1}
(b^{n-1}\ns{-n+1}b^n\ns{-n-1})\cdot a^{n-1})>.
\end{split}
 \end{align}
Using the fact that $N$ is AYD module we have,

\begin{align}\label{7}
 \begin{split}
&\eqref{6}=<\psi(b^n\ns{0}\ot b^0\ns{0}\odots b^{n-1}\ns{0})\cdot b^n\ns{-3}\; \mid\; \\
& \phi((S^{-1}(S(b^n\ns{-2})b^n\ns{-1} b^0\ns{-1}\cdots b^{n-1}\ns{-1})
b^n\ns{-4})(S^{-1}(b^n\ns{-n-2})\cdot a^n)\ot\\
&\ot S^{-1}(b^0\ns{1}\cdots b^{n-1}\ns{-1}b^n\ns{-5})\cdot a^0\odots
S^{-1}(b^{n-1}\ns{-n+1}b^n\ns{-n-4})\cdot a^{n-1})>.
\end{split}
 \end{align}
Using \eqref{compatibility1} and the facts that $\phi$ is $H$-linear and  $\mathcal{M}$ is AYD contramodule
 we see
\begin{align}
 \begin{split}
&\eqref{7}=<\psi(b^n\ns{0}\ot b^0\ns{0}\odots b^{n-1}\ns{0})b^n\ns{-1}\;\mid\; \\
& \phi(S^{-1}( b^0\ns{-1}\cdots b^{n-1}\ns{-1}) b^n\ns{-2}
S^{-1}(b^n\ns{-n-2})a^n\ot\\
&\ot S^{-1}(b^0\ns{1}\cdots b^{n-1}\ns{-1}b^n\ns{-3})
a^0\odots S^{-1}(b^{n-1}\ns{-n+1}b^n\ns{-n-2})a^{n-1})>\\
&=<\psi(b^n\ns{0}\ot b^0\ns{0}\odots b^{n-1}\ns{0})\;\mid\;
 \phi( S^{-1}( b^n\ns{-1}b^0\ns{-1}\cdots b^{n-1}\ns{-1} )a^n\ot\\
&\ot S^{-1}(b^0\ns{1}\cdots b^{n-1}\ns{-1})a^0\odots
S^{-1}(b^{n-1}\ns{-n+1})a^{n-1})>\\
&=\Phi(\phi\ot\psi)(a^n\al b^n \ot a^0\al b^0\odots
a^{n-1}\al b^{n-1})\\
&=\tau^{A\al B}\Phi(\phi\ot\psi)(a^0\al b^0\odots a^{n}\al b^{n}).
\end{split}
\end{align}
\end{proof}

\begin{theorem}\label{cup-theorem-2}
Let $H$ be a Hopf algebra,  $A$ be an $H$-module algebra, $B$ be an $H$-comodule
algebra, $(N, \mathcal{M})$ be a compatible pair of SAYD module and contramodule.
 Then the map $\hat\Phi:=\Phi\circ \overline{AW}$ defines a cup product:
\begin{equation} \label{cup-cont-algebra}
\hat\Phi:= \Phi\circ \overline{AW}:~^HHC^q(B, N)\ot HC^p_H(A,\mathcal{M})\ra HC^{p+q}(A\al B).
\end{equation}
\end{theorem}
\begin{proof}
The proof is similar to the proof of Theorem \ref{cup-theorem-1}.
Let $[\phi]\in ~^HHC^p(A,\mathcal{M})$ and $[\psi]\in HC^q_H(B,N)$. Without loss of
 generality one assumes that $\phi$ and $\psi$ are both cyclic cocycle, i.e,
  $\hb(\phi)=\vb(\psi)=0$,  $\hta(\phi)=(-1)^p\phi$ and
  $\vta(\psi)=(-1)^q\psi$ . This implies
 that $\psi\ot \phi$ is a $(b,B)$ cocycle in $\Tot(C_{\rm a-a}^{\ast,\ast})^{p+q}$.
 Hence $\overline{AW}(\psi\ot \phi)$ defines a class in
 $HC^{p+q}(D(C_{\rm a-a}^{\ast,\ast}))$.
 Finally,  since $\Phi$ is cyclic map, the transferred
  cochain   $\Phi(\overline{AW}(\psi\ot \phi))$ defines a class in $HC^{p+q}(A\al B)$.
\end{proof}
\section{Cup products for incompatible pairs}
In this section we generalize the cup products defined
in \eqref{cup-cont-coalgebra} and \eqref{cup-cont-algebra}
 to the case of incompatible coefficients. The target of
  the cup product in the new case is the ordinary cyclic
   cohomology of algebras with coefficients in a module
    produced out of the two incompatible coefficients.

Let $\mathcal{M}$ be a SAYD contramodule and let $N$ be a SAYD module over a Hopf algebra $H$. We define $L(N,\mathcal{M})$ to be the coequalizer
\begin{equation}
\xymatrix{
 N\ot _H \Hom(H, \mathcal{M})\ar@<.6 ex>[r] \ar@<-.6 ex>[r] & N\ot_H\mathcal{M}\ar[r] & L(N, \mathcal{M}),}
\end{equation}
where the  equalized maps are $n\ot m \mapsto n\ot \a(f)$, and $n\ot m\mapsto n\ns{0}\ot f(n\ns{-1})$.

\begin{remark}\label{remark}
{\rm
In fact $L(N,\mathcal{M})$ is the usual (contra)tensor product defined by Positselski \cite[page 96]{Pos09}. One considers $H$-coring $\mathcal{C}:=H\ot H$ with the usual coring structure and
 identifies $N$ with a left $H\ot H$-comodule  \cite{TB07}. In the same fashion one identifies $\mathcal{M}$ with  a right $\mathcal{C}$-contramodule. Then $L(N, \mathcal{M})$ is identified with the (contra)tensor $N\ot_\mathcal{C} \mathcal{M}$.}
\end{remark}
Now let $A$ and $C$ satisfy \eqref{0}...\eqref{3}. We
recall that the algebra  $B$ is $\Hom_H(C,A)$ with the convolution
multiplication and that  $C_{\rm a-c}^{\ast,\ast}$ is the bicocyclic module
defined in \eqref{c**}. We define
\begin{align}\label{acpsi-2}
&\tPsi: D(C_{\rm a-c}^{\ast,\ast})^q\ra \Hom(B^{\ot q+1},L(N,\mathcal{M})),\\\notag &
 \tPsi(\phi\ot
(n\ot c^0\odots c^q))(f^0\odots f^q)= n \ot_H  \phi(f^0(c^0)\odots f^q(c^q)).
\end{align}
By a similar argument as in the proof of Proposition \ref{cyclic-map}
one shows that $\td\Psi$ is a cyclic map. One proves the following theorem
 with a similar proof as of Theorem \ref{cup-theorem-1}
\begin{theorem} \label{cup-general-1}Let $H$ be a Hopf algebra,  $A$ be
an $H$-module algebra,  $C$ be an
 $H$-module coalgebra acting on $A$,  and   $(N, \mathcal{M})$ be a not necessarily
 compatible pair of SAYD
  module and contramodule over $H$. Then $\iota\circ\tPsi\circ\overline{AW}$
 defines a cup product in the level of cyclic cohomology:
$$\iota\circ\tPsi\circ\overline{AW}: HC^p_H(A,\mathcal{M})\ot HC^q_H(C,N)\ra HC^{p+q}(A, L(N,\mathcal{M})).$$
\end{theorem}
One notes that the range of this cup product is the
ordinary cyclic cohomology of the algebra $A$ with
coefficients in the vector space $L(N,\mathcal{M})$. One also
notes that if $(N,\mathcal{M})$ is compatible then  $E:L(N,\mathcal{M})\ra \Cb $
defined by $E(n,m)=<n,m>$ is a map of vector spaces. As a result we get a cyclic map
$$ \td E: \Hom(A^{\ot \ast},L(N,\mathcal{M}))\ra \Hom(A^{\ot \ast},\Cb), \quad \td E(\vp)=E\circ\vp.$$
 So we cover the old  cup product as   $\Psi=\tPsi\circ \td E$, where $\Psi$ is defined in \eqref{acpsi}.

Now let us generalize the other cup product for algebra-algebra in a similar fashion  as  the case of algebra-coalgebra.  Let $A$  be a
left $H $-module algebra,
 $B$ be a left $H $-comodule algebra,  $(N, \mathcal{M})$ be a   pair
  of SAYD module and contramodule over $H$, and   $A\al B$ be
  the crossed product algebra defined in \ref{crossed-algebra}.
   Let also $C_{\rm a-a}^{\ast,\ast}$ be
    the bicocyclic module defined in \eqref{bicyclic-algebra} .
     We define
 \begin{align}\label{abpsi-2}
&\td\Phi:D(C_{\rm a-a}^{\ast,\ast})^q\ra \Hom((A\al B)^{\ot q+1},L(N,\mathcal{M})),\\\notag
&\td \Phi(\psi\ot\phi)(a^0\al b^0\odots a^q\al b^q)\\\notag
&=\psi(b^0\ns{0}\odots b^q\ns{0})\ot_H \phi(  S^{-1}(b^0\ns{-1}\cdots
 b^q\ns{-1})a^0\odots S^{-1}(b^q\ns{-q-1})a^q).
\end{align}
Similarly we prove that $\td\Phi$ is cyclic and induces a
map on the level of cyclic cohomologies:
  \begin{theorem}\label{cup-theorem-2}
Let $H$ be a Hopf algebra,  $A$ be an $H$-module algebra,
$B$ be an $H$-comodule algebra, $(N, \mathcal{M})$ be a compatible
pair of SAYD module and contramodule. Then the map $\td\Phi\circ\overline{AW}$ defines a cup product in the level of cyclic cohomology:
$$\td\Phi\circ\overline{AW}:\;\;^HHC^p(B, N)\ot HC^q_H(A,\mathcal{M})\ra HC^{p+q}(A\al B, L(N,\mathcal{M})).$$
\end{theorem}

\end{document}